\documentclass[11 pt]{article}  

\usepackage[utf8]{inputenc}
\usepackage{amsmath,indentfirst,qsymbols,graphicx,psfrag,amsfonts,stmaryrd,hyperref,color,enumerate,amsthm}

\def \1{\textbf{1}}

\def \bla{\big\langle}
\def \bra{\big\rangle}

\def \ov#1{\overline{#1}}
\def \un#1{\underline{#1}}
\def \ND#1{{\sf ND}_{#1}}

\def \CCS{{\sf CCS}}
\def \CS{{\sf CS}}

\def \app#1#2#3#4#5{\begin{array}{rccl} #1:&#2&\longrightarrow&#3\\ &#4&\longmapsto&#5\end{array}}

\def \bar{\overline}
\def \Tri{{\sf Tri}}
\def \ba{\begin{align}}
\def \ea{\end{align}}
\def \be{\begin{eqnarray*}}
\def \ee{\end{eqnarray*}}
\def \ben{\begin{eqnarray}}
\def \een{\end{eqnarray}}

\def \beq{\begin{equation}}
\def \eq{\end{equation}}

\def \build#1#2#3{\mathrel{\mathop{\kern 0pt#1}\limits_{#2}^{#3}}}

\def \ba{{\bf a}}

\def \bx{\bar{x}}

\def \sx{{\sf x}}

\def \captionn#1{\begin{center}\begin{minipage}{15cm}\sf\caption{\small #1}\end{minipage}\end{center}}

\def \Tri{{\sf Tri}}

\def \eref#1{(\ref{#1})}

\def \Co{{\sf Cb}}

\def \l{\left}

\def \r{\right}
\def \sous#1#2{\mathrel{\mathop{\kern 0pt#1}\limits_{#2}}}
\def \sur#1#2{\mathrel{\mathop{\kern 0pt#1}\limits^{#2}}}
\def \eqd{\sur{=}{(d)}}
\def \w#1{\widetilde{#1}}

\def \Sym{{\sf Sym}}
\def \Sha{{\sf Sha}}
\newqsymbol{`B}{\mathcal{B}}
\newqsymbol{`E}{\mathbb{E}}
\newqsymbol{`I}{\mathbb{I}}
\newqsymbol{`N}{\mathbb{N}}
\newqsymbol{`O}{\Omega}
\newqsymbol{`P}{\mathbb{P}}
\newqsymbol{`Q}{\mathbb{Q}}
\newqsymbol{`R}{\mathbb{R}}
\newqsymbol{`W}{\mathbb{W}}
\newqsymbol{`Z}{\mathbb{Z}}
\newqsymbol{`a}{\alpha}
\newqsymbol{`e}{\varepsilon}
\newqsymbol{`o}{\omega}
\newqsymbol{`t}{\tau}
\newqsymbol{`w}{{\cal W}}
\def\cro#1{\llbracket#1\rrbracket}

\def\CP{{\sf CP}}
\begin{document}

\newtheorem{fig}{\hspace{2cm} Figure}
\newtheorem{lem}{Lemma}
\newtheorem{defi}[lem]{Definition}
\newtheorem{pro}[lem]{Proposition}
\newtheorem{theo}[lem]{Theorem}
\newtheorem{cor}[lem]{Corollary}
\newtheorem{note}[lem]{Note}
\newtheorem{conj}{Conjecture}
\newtheorem{Ques}{Question}
\newtheorem{rem}[lem]{Remark}
\renewcommand{\baselinestretch}{1.1}

 \begin{center}
 \huge\bf
Around Sylvester's question in the plane\\
 {\large \bf Jean-Fran\c{c}ois Marckert}
 \rm \\
 \large{CNRS, LaBRI, Universit\'e Bordeaux \\
  351 cours de la Libération\\
 33405 Talence cedex, France}
 \normalsize
 \end{center}

\begin{abstract}  Pick $n$ points $Z_0,...,Z_{n-1}$ uniformly and independently at random in a compact convex set $H$ with non empty interior of the plane, and let $Q^n_H$ be the probability that the $Z_i$'s are the vertices of a convex polygon. Blaschke 1917 \cite{Bla} proved that $Q^4_T\leq Q^4_H\leq Q^4_D$, where $D$ is a disk and $T$ a triangle. In the present paper we prove $Q^5_T\leq Q^5_H\leq Q^5_D$. One of the main ingredients of our approach is a new formula for $Q^n_H$ which permits to prove that Steiner symmetrization does not decrease $Q^5_H$, and that shaking does not increases it (this is the method Blaschke used in the $n=4$ case). We conjecture that the new formula we provide will lead in the future to the complete proof that  $Q^n_T\leq Q^n_H\leq Q^n_D$ , for any $n$.
\end{abstract}

\section{Introduction}

\paragraph{Notations and convention.} All the random variables (r.v.) in the paper are assumed to be defined on a common probability space $(\Omega,{\cal A},`P)$. The expectation is denoted by $`E$. The plane will be seen as $`R^2$ or as $\mathbb{C}$ and we will pass from the real notation (e.g. $(x,y)$) to the complex one (e.g. $\rho e^{i \theta}$ or $x+iy$) without any warning. For any $n\geq 1$, any generic variable name $z$, $z[n]$ stands for the $n$-tuple $(z_0,\dots,z_{n-1})$ and $z\{n\}$ for the set $\{z_0,\dots,z_{n-1}\}$. The set of compact convex subsets of $`R^2$ with non empty interior is denoted $\CCS$. For any $H\in \CCS$ and any $n\geq 0$, $`P_H^n$ is the notation for the law of $Z[n]$, a sequence of $n$ i.i.d. points taken under the uniform distribution over $H$. Last, we denote by $\cro{a,b}:=[a,b]\cup \mathbb{Z}$.

\subsection{The new result}

A $n$-tuple of points $\sx[n]$ of the plane is said to be in a convex position 
if $x\{n\}$ is the vertex set of a convex polygon. Denote by $\CP_{n}$ the set of $n$-tuples $x[n]$ in  a convex position.
Finally let 
\be
Q_H^n:=`P^n_H(\CP_n)=`P(Z[n]\in \CP_n),
\ee
where $Z[n]$ is  $`P_H^n$ distributed. The aim of this paper is to prove the following theorem:
\begin{theo}\label{theo:main}
 For any $H \in \CCS$,
\ben 11/36=Q^5_T\leq Q^5_H\leq Q^5_D=1-305/(48\pi)^2,
\een
with equality holds for the left inequality only when $H$ is a triangle, and in the right one, only when $H$ is an ellipse. 
 \end{theo}

Blaschke \cite{Bla} proved in 1917 that for any $H \in \CCS$,
\ben 
2/3=Q^4_T\leq Q^4_H\leq Q^4_D=1-35/(12\pi)^2, 
\een
with the equality cases are, for the left inequality only when $H$ is a triangle, and for the right one, only when $H$ is an ellipse. Roughly, the method of Blaschke relies on two ingredients: \\ 
A. there is an ``algebraic formula'' for $Q^4_H$~:
\ben\label{eq:q4}
Q^4_H= 1-4`E( {\sf Area}(Z_0,Z_1,Z_2))=1- 2`E(|\det(Z_1-Z_0,Z_2-Z_0)|)
\een
where ${\sf Area}(Z[3])$ is the non negative area of the triangle $Z[3]$ under $`P_{H}^3$ (since 4 points are not in a convex position, if one of them lies in the triangle formed by the 3 other ones). \\
B. Steiner symmetrization and shaking (see definitions below) have the following property:\\
\indent a. If $H^\Sym$ (resp. $H^\Sha$) are obtained from $H$ by a Steiner symmetrization (resp. a shaking) with respect to the $x$-axis, then
\ben\label{eq:symm}
Q^4_{H^\Sym}\geq Q^4_{H},~~ Q^4_{H^\Sha}\leq Q^4_{H}, \een
with equality only in some identified special cases;\\
\indent b. For any $H_0\in \CCS$, there exists a sequence of lines $(\Delta_i, i\geq 1)$, so that for $H_{i+1}$ obtained from $H_i$ by Steiner symmetrization (resp. shaking) with respect to $\Delta_{i+1}$, the sequence $(H_n)$ converges to a disk (resp. to a triangle) for the Hausdorff distance (see  Klartag \cite{Klar} and Campi, Colesanti and Gronchi \cite{CCG} for modern and general treatments).
\begin{rem}\label{rem:SSS} Formula \eref{eq:symm} is only needed for Steiner symmetrization and shaking with respect to the x-axis, since rotations conserve uniform distributions and convexity. 
\end{rem}
In the present paper we use the same methodology. Hence, we need an ``algebraic formula'' for $Q^5_H$, and prove that it satisfies the analogous of B.a.:
\begin{theo}\label{theo:ineq} For $n=4$ and $n=5$,
\ben
Q^n_{H^\Sym}\geq Q^n_{H},~~ Q^n_{H^\Sha}\leq Q^n_{H}.
\een
\end{theo}
We will provide a slightly different proof than that of Blaschke for the case $n=4$.\par
 
Maybe this is the right place to discuss the presence of quotation marks around ``algebraic formula'' here. In fact the determinant is algebraic in the coordinates of the $Z_j'$s but \eref{eq:q4} is more complex than this since it involves an absolute value, and an expectation. In the case $n=4$, this expectation is a triple integral over $H$ that could be as unpleasant as one could imagine.

To circumvent this problem the integrals are not really directly compared. What are compared and some quantities under the integral signs: writing the integrals with respect to $x[3],y[3]$, the coordinates of $z[3]$, and integrating only partially, that is according to certain of these variables only (e.g. $y[3]$), keeping the integrals with respect to the other variables. This quantity ``below a certain number of integral signs'' has also an algebraic form, since when one integrates $\det(z_1-z_0,z_2-z_0)$ according to the variables $y_0,y_1,y_2$ we still get an algebraic result, in fact a polynomial depending on the $x_i$'s and on the maximal and minimal ordinates of the points of $H$ in each of the slices at abscissas $x_0$, $x_1$ and $x_2$. Of course, the fact that this is the absolute value of the determinant which matters, brings some extra complications. 

One of the main advances in the paper is a new ``algebraic'' formula for $Q^n_H$ (including $Q^4_H$) which avoids absolute values, and which is given in terms of $n+2$ real integrals of a polynomial. The comparison of the polynomials appearing
when computing $Q^n_H,Q^n_{H^{\Sym}}$ and $Q^n_{H^{\Sha}}$ will give the expected result, when $n=4$ and $n=5$. We were not able to go further, because of the complexity of the involved polynomials.

\subsection{Related results}

The problem of determining $Q^n_H$ goes back to a question (badly) posed by Sylvester \cite{SYL}. Finally, the question was to show that the map $H\mapsto Q^4_H$ takes its maximum on $\CCS$ when $H$ is a disk and its minimum when $H$ is a triangle, Theorem finally proved by Blaschke \cite{Bla} (see Pfeifer \cite{Pfe} for historical notes). 
Recently some advances have been made on the exact computations of $Q^n_H$: 
Valtr \cite{Valtr-P,Valtr-T} showed that if $S$ is a square (or a non flat parallelogram) and  if $T$ is a non flat triangle then, for $n\geq 1$, 
\ben
Q_S^n =\l(\frac{\binom{2n-2}{n-1}}{n!}\r)^2,~~~ Q_T ^n= \frac{2^n (3n-3)!}{(n-1)!^3(2n)!}.
\een
Buchta \cite{BC2} goes further and gives an expression for $Q_S^{n,m}$ and $Q_T^{n,m}$, the probability that $m$ points exactly among the $n$ random points are on the boundary of the convex hull. 
The author of the present paper gives a formula for $Q^n_D$ (and $Q^{n,m}_D$) in the disk case \cite{MAR}.
 
The literature concerning  the question of the number of points on the convex hull for i.i.d. random points taken in a convex domain is huge. We won't make a survey here but rather refer the reader to Reitzner \cite{Rei} and Hug \cite{Hug} to have an overview of the topic.

As far as we know, Blaschke result has not been extended in the direction we propose here, but rather, in the multidimensional case, where o Blaschke \cite{Bla} proved that  
\ben\label{eq:blas-gene-d}
 Q^{d+2}_{K}\leq Q^{d+2}_{B_d},~~~\textrm{ for any } K\in {\cal K}_d,
\een
where ${\cal K}_d$ is the set of compact convex bodies in $\mathbb{R}^d$ with non empty interior, $B_d$ is the unit ball in $\mathbb{R}^d$.
The inequality $Q^{d+2}_{\Delta_d}\leq Q^{d+2}(K)$ for any $K\in {\cal K}_d$, where $\Delta_d$ is the simplex in $\mathbb{R}^d$ is still a conjecture. Milman \& Pajor \cite[Prop. 5.6]{MP} established that if it holds, then the hyperplane conjecture (or slicing problem) holds true: there exists a universal constant $c>0$ such that for every $d$ and for every convex body $K$ of volume one in $\mathbb{R}^d$ there exists an hyperplane $H$ such that $|K\cap H|\ge c$. This connection is another justification for our work since a right understanding of the 2-D case can be a step in the right direction.

\subsection{Content of the paper}

Most of the paper is devoted to proving Theorem \ref{theo:ineq}.
In section \ref{sec:PTineq}, we recall what are Steiner symmetrization and shaking with respect to the $x$-axis. Take $Z[n]$ under $`P_K^n$ for some $K\in\CCS$, and let $Z_j=(X_j,Y_j)$.
A property of the Steiner symmetrization and shaking with respect to the $x$-axis is that the distribution of the abscissas $X[n]$ is the same when $K$ is $H$, $H^{\Sym}$ or $H^{\Sha}$. 

We then prove the stronger Prop. \ref{pro:comp} which asserts that  Theorem \ref{theo:ineq} holds when we condition on $X[n]=x[n]$ for any sequence $x[n]$. Conditional on $X[n]=x[n]$, $Z_j$ is uniform in the vertical segment $V_K(x_j)=\{(x,y)\in K, x=x_j\}$ for $K$ depending on the case of interest. We then apply a ``normalisation procedure'' in Section \ref{sec:NVP}, which amounts to sending the three collections of segments $(V_K(x_j),0\leq j \leq n-1)$ for 
$K\in\{H,H^{\Sym},H^{\Sha}\}$ by three inversible affine maps which preserves verticality (see Defi. \ref{def:PV}) onto three families of segments $\bar{V},\bar{V}^{\Sha},\bar{V}^{\Sym}$ so that the first (and also the last) segments of these families coincides. 
We then provide an algebraic formula for the probability $\langle S_0,\cdots,S_{n-1}\rangle$ that $n$ independent random points $Z[n]$  are in a convex position, where $Z_j$ is taken uniformly in the vertical segment $S_j$. 
In Lemma \ref{lem:cond_boundary} we condition on the two $Z_j$ with maximum and minimum abscissas, say, $z_0$ and $z_{n-1}$. Appears then that the remaining $Z_j$ are above or under the line $(z_{0},z_{n-1})$, and are uniform on the part of the segments $S_j$ in which they lie. 
The  structure which appears is that of two combs (Defi. \ref{def:Pcomb}), one above and one below the line $(z_0,z_{n-1})$, with some random points on each tooth. In Prop. \ref{pro:dec_comb}, the general formula for $\langle S_0,\cdots,S_{n-1}\rangle$ in terms of $\langle \Co[x_j,\ell_j]_{1\leq j \leq m}\rangle$ which is the quantity relative to the probability that some random points taken on an orthogonal comb are in a convex position (where the tooth length and position are encoded by the $[x_j,\ell_j]$'s) are given.

Prop. \ref{pro:comb-dec} provides a combinatorial-like decomposition for  $\langle \Co[x_j,\ell_j]_{1\leq j \leq m}\rangle$ implying that this quantity is a rational fraction in the coordinates of the comb teeth extremities.
Section \ref{sec:opti} is devoted to ending the proof of Theorem \ref{theo:ineq} by optimizing these formulas when $n=4$ and $n=5$.
In Section \ref{sec:CGR} some additional elements on the algebraic structure of $\langle S_0,\cdots,S_{n-1}\rangle$ are given. We end by giving a ``combinatorico-geometric'' representation of $\langle \Co[x_j,\ell_j]_{1\leq j \leq m}\rangle$.

\section{Proof of Theorem \ref{theo:ineq}: preliminaries}
\label{sec:PTineq}

\subsection{Abscissas fibration}
\label{sec:AF}

A well known property which comes from that inversible affine transformations conserve convexity and uniformity, is the following fact:
\begin{lem}\label{lem:eto} For any inversible affine map $A$ of $\mathbb{R}^2$, for any $H\in\CCS$, we have $Q^n_{A(H)}=Q^n_{H}$.
\end{lem}
A consequence of this is that we can prove Theorem \ref{theo:main} only for convex bodies with area 1.

As represented on Fig. \ref{fig:com}, for any $H$ with area 1, denote by $x_{\min}(H)=\min\{x:(x,y)\in H\}$ and $x_{\max}(H)=\max\{y:(x,y)\in H\}$ the minimum and maximum abscissas of $H$ and let,
\be
\ov{y}_x(H)&:=& \inf\{y: (x,y)\in H \},~~~ \textrm{ for }x\in[x_{\min}(H),x_{\max}(H)]\\
\un{y}_x(H)&:=&\sup \{y: (x,y)\in H \}~~~ \textrm{ for }x\in[x_{\min}(H),x_{\max}(H)].
\ee
The width function $W_H:\mathbb{R}\to\mathbb{R}$ is defined by
\[W_H(x)=\l(\ov{y}_x(H)- \un{y}_x(H)\r) \1_{[x_{\min}(H),x_{\max}(H)]}(x).\]
The vertical segment intersecting $H$ at abscissa $x$ is denoted
\[V_H(x)=x+i\l[\un{y}_x(H),\ov{y}_x(H)\r].\]
The law of the abscissa $X$ of a uniform point $(X,Y)$ taken in $H$ has density $W_H$.
\begin{figure}[ht]
\centerline{\includegraphics{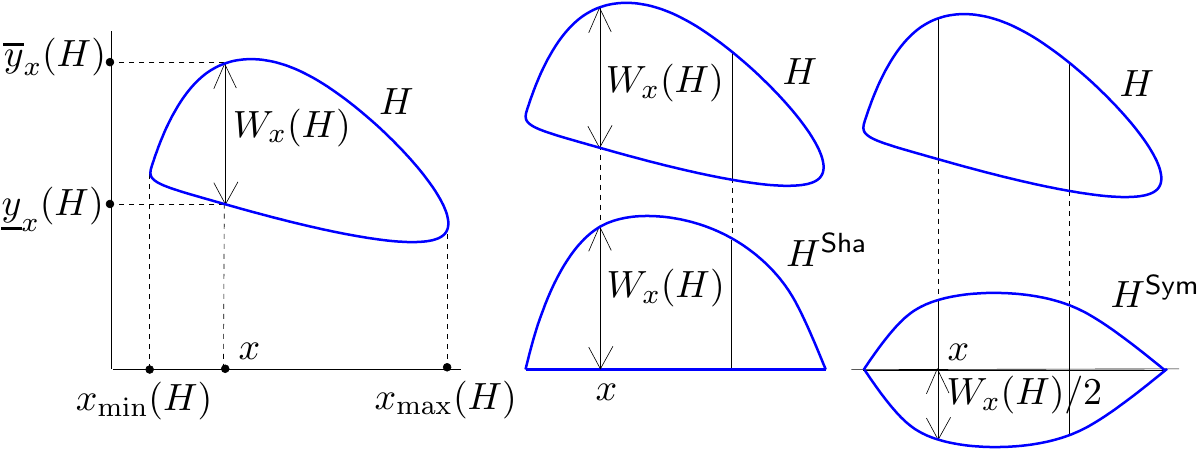}}
\captionn{\label{fig:com}Steiner symmetrization and shaking with respect to the $x$-axis.} 
\end{figure}

\begin{note}
Instead of taking $n$ points at random in $H$, in the sequel we will take $N+2$ points at random! Of course for $n=N+2$ this is equivalent, but it will be useful in our decompositions/recurrences to have a point with rank 0, and one with rank $N+1$, to get simpler formula.
\end{note}
Let $Z[N+2]$ be taken under $P_H^{N+2}$ and let $(X_j,Y_j)$ be the coordinates of $Z_j$ in the plane. Consider $\tau$ the a.s. well defined permutation in the symmetric group ${\cal S}(\cro{0,N+1})$ such that 
\be
X_{\tau(0)}\leq \cdots \leq X_{\tau(N+1)}.
\ee
By symmetry, the permutation $\tau$ is uniform in ${\cal S}(\cro{0,N+1})$ and independent from the set of values $\{X_j,0\leq j \leq N+1\}$. The density of $X_\tau:=\l(X_{\tau(j)},0\leq j \leq N+1\r)$ on $\mathbb{R}^{N+2}$ is
\[f_H(x[N+2])=(N+2)! \l(\prod_{j=0}^{N+1} W_H(x_j) \r)   \1_{\ND{N+2}}(x[N+2]) \]
where, for any $n\geq 1$, $\ND{n}=\{x[n]~: x_0< \cdots < x_{n-1}\}$ is the set of non decreasing sequences with $n$ elements.
Conditional on $(X_{\tau(j)},0\leq j \leq N+1)=x[N+2]$, the variables $Z_{\tau(0)},\cdots,Z_{\tau(N+1)}$ are independent, and $Z_{\tau(j)}$ is uniform on $V_H(x_{j})$. We introduce a crucial object of the paper:
\begin{defi}
Consider $N+2$ (vertical or not) segments $(S_0,\cdots,S_{N+1})$ of the plane and $U[N+2]$ a $N+2$-tuple of independent r.v. where $U_j$ is uniform on $S_j$. We denote by 
\ben
\big\langle S_0,\cdots,S_{N+1} \big\rangle:= P(U[N+2]\in \CP_{N+2}),
\een
the probability that the $U_j$'s are in a convex position. 
\end{defi}
To compute  $Q_H^{N+2}$, one can condition on the value of  $X_\tau$, from what we see that
\ben\label{eq:des}
Q_H^{N+2}&=&\int_{\ND{N+2}}  \bla V_H(x_0),\cdots,V_H(x_{N+1})\bra\, f_H(x[N+2])\, d{x_0}\cdots d{x_{N+1}}.  
\een
Consider $H^\Sym$ and $H^\Sha$ the convex bodies obtained from $H$ by Steiner symmetrization and shaking with respect to the x-axis~:
\be
H^\Sym&=&\l\{(x,y)~: x_{\min}(H)\leq x \leq x_{\max}(H), |y|\leq W_H(x)/2 \r\}, \\
H^\Sha&=&\l\{(x,y)~: x_{\min}(H)\leq x \leq x_{\max}(H), 0\leq y\leq W_H(x) \r\}. 
\ee
Since the width functions $W_H, W_{H^\Sha}$ and $W_{H^\Sym}$ coincide, it can be deduced that
\begin{equation}
f_H=f_{H^\Sha}=f_{H^\Sym}.
\end{equation}

In view of \eref{eq:des}, Theorem \ref{theo:main} appears to be a consequence of the following proposition:
\begin{pro}\label{pro:comp}
For $H\in \CCS$, $N\in\{2,3\}$, and $x[N+2]\in \ND{N+2}\cap [x_{\min}(H),x_{\max}(H)]^{N+2},$
\ben\label{eq:ineq}
\bla V_H^{\Sha}(x_0),\cdots,V_H^{\Sha}(x_{N+1})\bra\leq \bla V_H(x_0),\cdots,V_H(x_{N+1})\bra \leq \bla V_H^{\Sym}(x_0),\cdots,V_H^{\Sym}(x_{N+1})\bra.
\een
\end{pro}
\begin{note} In fact, to get Theorem \ref{theo:main}, we need also to treat the equality case. In fact, a consequence of our proof below, is that equality \eref{eq:ineq} holds iff :\\
-- for the left one, if the bottom (or top) points of all segments $V_H(x_j)$ are aligned,\\
-- for the right one, if  the $V_H^{\Sym}(x_j)$ are symmetric with respect to $x$-axis.\\
Since to get $Q^n_H$ we integrate afterward against $f_H(x[N+1])$, we have $Q^n_{H^{\Sym}}=Q^n_H$ (resp. $Q^n_{H^{\Sha}}=Q^n_H$) only when $H$ is symmetric with respect to $x$-axis (resp. $x\mapsto \ov{y}_H(x)$ or  $x\mapsto \un{y}_H(x)$ are linear). By standard arguments, this additional details suffices to prove Theorem \ref{theo:main} from Prop. \ref{pro:comp}.
\end{note}
To prove Proposition \ref{pro:comp}, we will need several steps. First, we will transport the segments of interest at a more favourable places in the plane. 
\subsection{Normalized version of the problem}
\label{sec:NVP}
Even if the situation is a bit different to that of Lemma \ref{lem:eto},  it is easy to see that again, 
\begin{lem}For any  inversible affine map $A$ of the plane, any sequence of segments $S[N+2]$,
\[\bla S_0,\cdots,S_{N+1}\bra=\bla A(S_0),\cdots,A(S_{N+1})\bra.\]
\end{lem}
We will use affine maps that preserves verticality:
\begin{defi}\label{def:PV}
An affine transformation $A$ of the plane is said to preserve verticality, if $A(x,y)-A(x,y')=(0,y-y')$ for any $x,y,y'$. The set of inversible such maps is denoted ${\sf PV}$. 
\end{defi}
Any $A\in{\sf PV}$ is an affine map of the form $A(x,y)= (c_1+c_2x, y+c_3x+c_4)$ for some $c_1,c_2,c_3,c_4$ and $c_2\neq 0$.
Take two pairs of points  ${\sf pair}^{(a)}:=(z_0^{(a)},z_1^{(a)})$, for $a\in\{0,1\}$. There is a unique  element $A\in{\sf PV}$ such that $A({\sf pair}^{(0)})={\sf pair}^{(1)}$ as soon as, for $a\in\{0,1\}$, both points $z_1^{(a)},z_0^{(a)}$ are not on the same vertical line.

By three affine maps preserving verticality, we send the three families of segments appearing in \eref{eq:ineq} onto three families of segments having the same extreme segments:
\begin{defi}Let $S[N+2]$ be a sequence of $N+2$ vertical segments at successive non decreasing abscissas $x[N+2]\in \ND{N+2}$. We call {\it normalizing map } of $S[N+2]$ the map $A\in {\sf PV}$ (depending on $S[N+2]$) which sends the middle $m(S_0)$ of $S_0$ at $0$ and the middle $m(S_{N+1})$ of $S_{N+1}$ at $1$~:
\[A(x+iy)= h(x)+i(y -m(S_0)-(m(S_{N+1})-m(S_0))h(x)),\]
for  $h(x)=(x-x_0)/(x_{N+1}-x_0)$. 
We call the sequence of segments $(A(S_j),0\leq j \leq N+1)$, the normalized version of $S[N+2]$. We finally call ${\sf Normal}$ the map which sends a sequence of segments onto its normalized version (see illustration on Fig. \ref{fig:syme2}).
\end{defi}
\begin{figure}[ht]
\centerline{\includegraphics{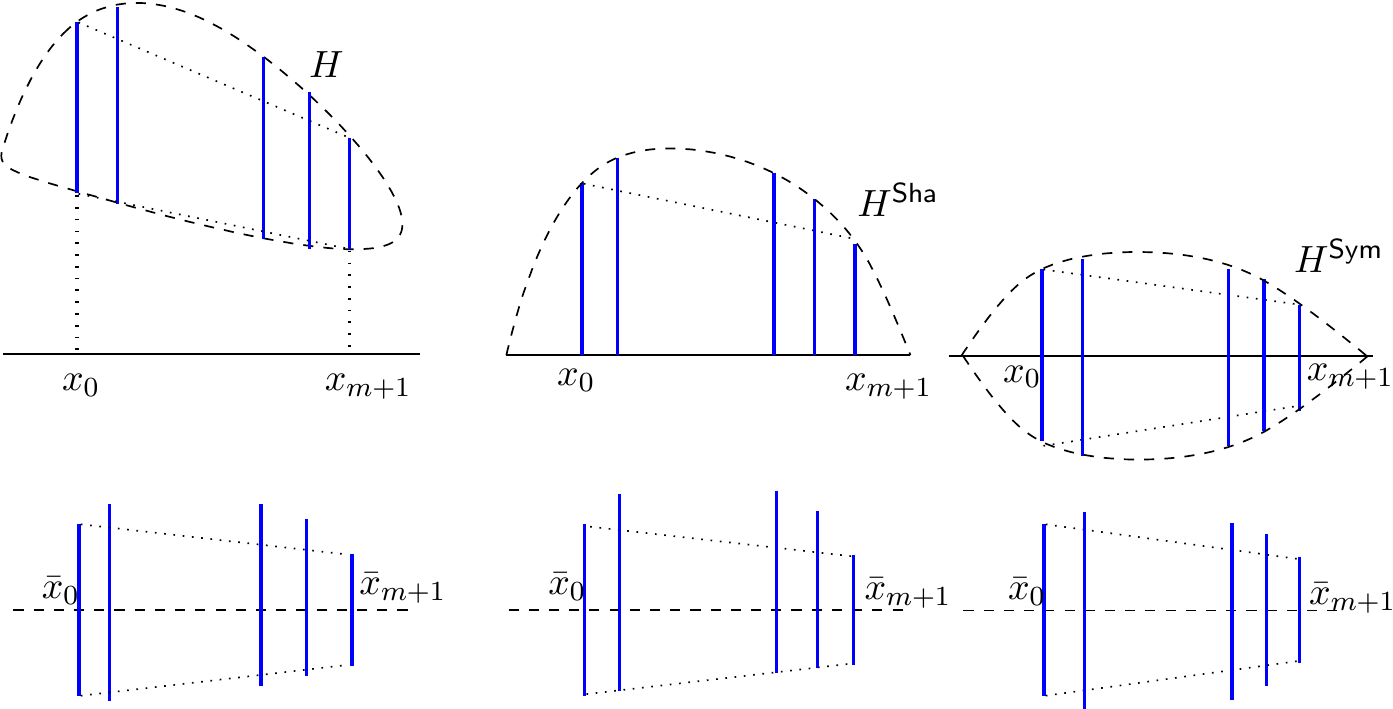}}
\captionn{\label{fig:syme2}Steiner symmetrization and shaking with respect to the $x$-axis of the vertical blue lines, followed by normalisation of a family of segments. The trapezoid (discussed in point (d)) for each sequence of segments coincides (for the pictures at the bottom). Notice that the abscissas of the normalized sequences of segments start at $\bar{x}_0=0$, and end at $\bar{x}_{m+1}=1$.  In the shaking case, the segments raise at the bottom line of the trapezoid.} 
\end{figure} 
Since the three families of segments $(V_H(x_j),0\leq j \leq N+1)$, $(V_{H^\Sym}(x_j),0\leq j \leq N+1)$ and $(V_{H^\Sha}(x_j),0\leq j \leq N+1)$ possess segments at the same abscissas with the same length, but with different ordinate, normalisation keeps this fact. After normalisation the new abscissas are
 $\bx_j=h(x_j),$ with $\bar{x}_0=0,\bar{x}_{N+1}=1$ (see Fig. \ref{fig:syme2}), and the first segment $v_0$ and  $v_{N+1}$ the last one of each family coincide. Consider  $\Gamma$ the trapezoid with sides ${v}_0,{v}_{N+1}$. Its top side is ${\sf TL}:=\{(x,L(x)): x\in[0,1]\}$ and its bottom side is  ${\sf TL}:=\{(x,-L(x)): x\in [0,1]\}$ where $L(x):=(W_H(x_0)+(W_H(x_{N+1})-W_H(x_0)){x})/2$ (see again  Fig. \ref{fig:syme2}).
Set $L_j=L(\bx_j)$, $\lambda_j= W_H(x_j)/2 -L_j$, and $\alpha_j=\ov{y}_H(x_j)-L_j-\lambda_j$ (so that $\lambda_0=\lambda_{N+1}=\alpha_0=\alpha_{N+1}=0$).\par
For a vector $\beta[N+2]$ such that $\beta_0=0$, $\beta_{N+1}=0$. Set 
\begin{equation}\label{eq:GF}
{V}^{\beta[N+2]}[N+2]:=\l(\bar{x}_j+i[-L_j-\lambda_j+\beta_j,L_j+\lambda_j+\beta_j], j\in\cro{0,N+1}\r).
\end{equation}
We call $\beta[N+2]$ the {\it symmetry defect}. The following Lemma whose proof is a simple exercise, gives a representation of the normal versions of the three families of segments of interest in terms of $V^{\beta[N+2]}[N+2]$: in the symmetric case $\beta[N+2]=0[N+2]$ (the null vector in  $\mathbb{R}^{N+2}$), in the (normalized) shaking case, the symmetry defect $\beta[N+2]$ ``is maximum'', the general case lying in between these extreme cases.
\begin{lem}
We have
\ben
{\sf Normal}(V_H(x_j),0\leq j \leq N+1)&=&{V}^{\beta[N+2]}[N+2] ,\\
{\sf Normal}(V_{H^\Sym}(x_j),0\leq j \leq N+1)&=&{V}^{0[N+2]}[N+2],\\
{\sf Normal}(V_{H^\Sha}(x_j),0\leq j \leq N+1)&=&{V}^{\lambda[N+2]}[N+2].
\een
\end{lem}
Let us examine a bit the possible symmetry defects. Consider the slope differences
\ben\label{eq:inefund}
p_j(\lambda)&:=&\frac{\Delta \lambda_j}{\Delta \bx_j}- \frac{\Delta \lambda_{j+1}}{\Delta \bx_{j+1}},~~~j\in\cro{1,N}\\
q_j(\beta)&:=&\frac{\Delta \beta_j}{\Delta \bx_j}- \frac{\Delta \beta_{j+1}}{\Delta \bx_{j+1}},~~~j\in\cro{1,N}
\een
(where $\Delta y_j:= y_j-y_{j-1}$). They always satisfy
\ben\label{eq:pq}
p_j(\lambda) \geq 0, |q_j(\beta)|\leq p_j(\lambda), ~~\textrm{ for any } j\in\cro{1,N},\\
\lambda_j \geq |\beta_j| ~~\textrm{ for any } j\in\cro{0,N+1}.
\een
The set of symmetry defects compatible with $\lambda[N+2]$ is denoted by
\begin{equation}
{\sf Compa}({\lambda[N+2]})=\{\beta[N+2]:|q_j(\beta)|\leq p_j(\lambda), \beta_0=\beta_{N+1}=0, |\beta_j|\leq \lambda_j\}.
\end{equation} 
Since $\lambda_0=\lambda_{N+1}=0$, given some elements $p=(p_j,j \in \cro{1,N})$, one can find $\lambda[N+2]$ so that $p=p(\lambda)$, and $\lambda_0=\lambda_{N+1}=0$~:
\ben\label{eq:lambdavvp}
\lambda_m= \frac{\sum_{j= m }^N p_j (x_j-x_{N+1})(x_0-x_m)+\sum_{j=1}^{m-1}p_j(x_m-x_{N+1})(x_0-x_j)}{x_{N+1}-x_0}, ~~ m \in\cro{1,N}
\een
where $x_0=0,x_{N+1}=1$. The same formula holds for $\beta$ in terms of $q(\beta)$.
 
All the discussions above imply that Prop. \ref{pro:comp} as well as Theorem \ref{theo:ineq} is a consequence of:
\begin{pro}\label{pro:red} For any $N\in\{2,3\}$, any $x[N+2]\in \ND{N+2}$, $x_0=0,x_{N+1}=1$, any sequence
 $\lambda[N+2]\in[0,+\infty)^{N+2}$, any linear map $L_j=A+Bx_{j}$ with $(A,B)\in[0,+\infty)$
the map 
\ben
\app{Q_{\lambda[N+2]}}{{\sf Compa}({\lambda[N+2]})}{[0,1]}{\beta[N+2]}{\bla {V}^{\beta[N+2]}[N+2] \bra}
\een
reaches its maximum when $\beta[N+2]=0[N+2]$ and its minimum when $\beta[N+2]=\pm \lambda[N+2]$.
\end{pro}
Write for short $V^\star$ instead of  $V^{\beta[N+2]}([N+2])$ and set
\ben
\label{eq:t1}\ell^+_j&=&L_j+\lambda_j+\beta_j, ~~j \in \cro{0,N+1},\\
\label{eq:t2}\ell^-_j&=&L_j+\lambda_j-\beta_j, ~~j \in \cro{0,N+1},\\
\label{eq:t3}w_j&=&\ell^+_j+\ell^-_j=2\lambda_j+2L_j, ~~j \in \cro{0,N+1}.
\een
Consider for $j\in\cro{0,N+1}$ some r.v. $Z_j=({x_j},Y_j)$ uniform r.v. on ${V}^{\star}_j$. Let
\[U_j:= Y_0+{x_j}(Y_{N+1}-Y_0), ~~j \in \cro{0,N+1}\] 
the ordinate of the line $(Z_0,Z_{N+1})$ at abscissa ${x_j}$ (where it intersects ${V}^{\star}_j$).
Let
\[{\sf Abo}:=\{j \in \cro{1,N}~: Y_j \geq U_j\}, \] 
the set of indices corresponding to the $Z_j$'s above the line $(Z_0,Z_{N+1})$. 
For any $A\subset \cro{1,N}$, 
\[P({\sf Abo}=A~|~Y_0=u_0,Y_{N+1}=u_{N+1})=\prod_{j \in A} \frac{\ell_j^+-u_j}{w_j} \prod_{j \in \complement A} \frac{\ell_j^-+u_j}{w_j},\]
where $\complement A := \cro{1, N}\setminus A$. Now, conditionally on $(Y_0,Y_{N+1},{\sf Abo})=(u_0,u_{N+1},A)$, set
\begin{equation}\label{eq:uj}
u_j= u_0+x_{j}(u_{N+1}-u_0).
\end{equation}
We will use the following notation: when $A$ is a set of integers (indices), then for any generic variable $y$, $y\{A\}$ is the set $\{y_j,j\in A\}$, and $y(A)$ denotes the tuple $(y_j, j\in A)$ where the indices $j$ are sorted according their natural order in $A$. 

The proof of the following lemma is immediate (see Fig. \ref{fig:syme3}):
\begin{lem}\label{lem:cond_boundary}
Conditionally on $(Y_0,Y_{N+1},{\sf Abo})=(u_0,u_{N+1},A)$:\\
(a) the set $Z\{A\}$ is independent of $Z\{\complement A\}$,\\ 
(b) the r.v. $Z_j(A)$ are independent, and $Z_j$ is uniform on $x_j+i[u_j,\ell^+_j]$,\\
(c) the r.v. $Z(\complement A)$ are independent, and $Z_j$ is uniform on $x_j+i[-\ell^-_j,u_j]$,\\
(d) the points $Z[N+2]$ are in convex position iff both
$\{(0,u_0),(1,u_{N+1})\}\cup Z\{A\}$ and $\{(0,u_0),(1,u_{N+1})\}\cup Z\{\complement A\}$ are in convex position.
\end{lem}
\begin{figure}[ht]
\centerline{\includegraphics{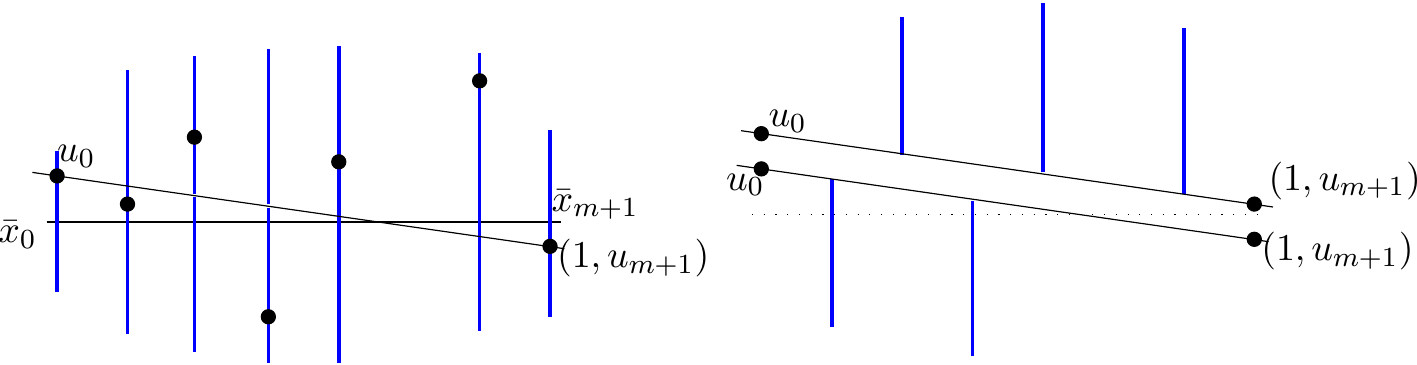}}
\captionn{\label{fig:syme3} On the left picture: the two extreme points being fixed $(Z_0,Z_{m+1})=(u_0,(1,u_{m+1}))$ being fixed, the set of indices of the points above the line is ${\sf Abo}=\{2,4,5\}$. Now, on the right one: Conditional on  ${\sf Abo}=\{2,4,5\}$, there are some uniform points above the line in each of the part of the segments $2,4$ and $5$, and some uniform points under the line in the part of the segments $1$ and $3$. Appears, the two inclined combs above and below the line. The $m+2=7$ points are in a convex positions, if in the superior comb, the 5 points are in a convex positions, and below, the 4 points also.  } 
\end{figure}
Conditionally on $(Z_0,Z_{N+1},{\sf Abo})=(u_0,(1,u_{N+1}),A)$  the points $(Z_0,Z(A),Z_{N+1})$ are uniform on a collection of segments, collection which may be seen as an inclined comb, the shaft being not orthogonal to the teeth (see Fig. \ref{fig:syme3}):
\ben
I_0&=&i[u_0,u_0],\\
I_{N+1}&=&1+i[u_{N+1},u_{N+1}],\\
I_j&=&x_j+i[u_j,\ell^+_j],~~j\in A,
\een
so that
\[P((Z_0,Z(A),Z_{N+1})\in \CP_{2+\#A}~|~(Z_0,Z_{N+1},{\sf Abo})=(u_0,(1,u_{N+1}),A))=\bla I_0, I(A),I_{N+1}\bra.\]
\begin{defi}\label{def:Pcomb} For any $x,\ell$ in $\mathbb{R}$, denote by $\CS$ the ``canonical segment''
\[\CS[x,\ell]=x+i[0,\ell].\]
For any $m\geq 0$, for any sequence $x[m+2]\in\ND{m+2}$ with $x_0=0,x_{m+1}=1$ and $\ell_0=0,\ell_1,\dots,\ell_{m}\geq 0,\ell_{m+1}=0$, we let $\Co[x_j,\ell_j]_{1\leq j \leq m}$ be the orthogonal comb (illustrated on Fig. \ref{fig:com})
\[\Co[x_j,\ell_j]_{1\leq j \leq m}:=\l(\CS[x_j,\ell_j],0\leq j \leq m+1\r).\]
Notice that the two extreme segments are reduced to a single point.
\end{defi}
\begin{figure}[ht]
\centerline{\includegraphics{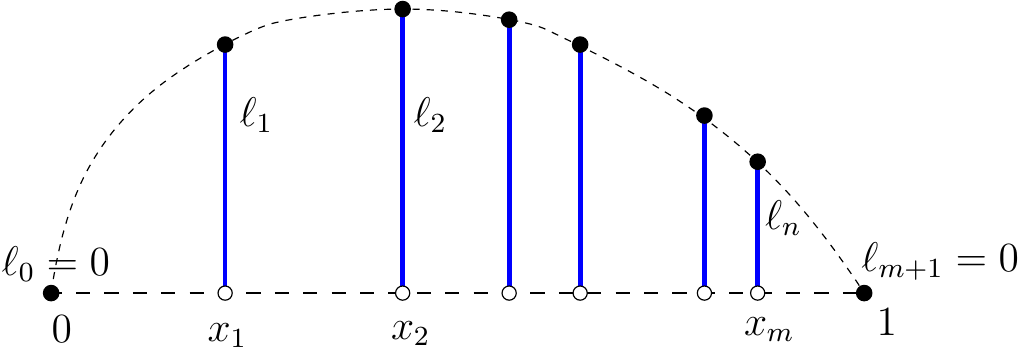}}
\captionn{\label{fig:com}The comb $\Co[x_j,\ell_j]_{1\leq j \leq m}$.} 
\end{figure}
One has
\ben\label{eq:rece}
\bla I_0,I(A),I_{N+1}\bra=\l\langle\Co\l[x_j,\ell_j^+-u_j\r]_{j\in A}\r\rangle,
\een
since the orthogonal comb  is the image of the inclined one by the affinity $(a,b) \mapsto (a,b-u_0-a(u_{N+1}-u_0))$. A consequence of Lemma \ref{lem:cond_boundary} is the following Proposition:
\begin{pro}\label{pro:dec_comb} 
We have, for $u_j= u_0+{x_j}(u_{N+1}-u_0)$,
 \begin{eqnarray} \label{eq:form}
 \bla   {V}^{\beta[N+2]}[N+2] \bra&=&\int_{-L_0}^{L_0}\int_{-L_{N+1}}^{L_{N+1}} \sum_{A\subset\{1,\cdots,N\}}   
 \prod_{j \in A} \frac{\ell_j^+-u_j}{w_j} \prod_{j \in \complement A} \frac{\ell_j^-+u_j}{w_j}  \\
 &&\times\l\langle\Co[x_j,\ell_j^+-u_j]_{j\in A}\r\rangle \l\langle\Co[x_j,\ell_j^-+u_j]_{j\in \complement A}\r\rangle
 \,\frac{du_{0}\,du_{N+1}}{w_0w_{N+1}}.
 \end{eqnarray}
\end{pro}
Next proposition provides the crucial close formula for $\l\langle\Co[ x_j,\ell_j]_{1\leq j\leq m}\r\rangle$, which appears to be decomposable (see also Fig. \ref{fig:comb2}), and to be a rational fraction in the $x_j$'s and the $\ell_j$'s
\begin{pro}\label{pro:comb-dec}
We have  $\l\langle\Co[\varnothing]\r\rangle=1$, $\l\langle\Co[x_1,\ell_1]\r\rangle=1$, and for $m\geq 2$
\ben\label{eq:qdsd}
\l\langle\Co\Big[x_j,\ell_j\Big]_{1 \leq j \leq m}\r\rangle& = & 
\frac1m \sum_{j=1}^m  \l(\prod_{1\leq k<j}\frac{\ell_k-\frac{x_k}{x_j}\ell_j}{\ell_k}\r)
\l(\prod_{j<k\leq m}\frac{\ell_k-\frac{1-x_k}{1-x_j}\ell_j}{\ell_k}\r)\\
&\times& \l\langle\Co\l[\frac{x_k}{x_j},\ell_k-\frac{x_k}{x_j}\ell_j\r]_{1\leq k < j}\r\rangle
\l\langle\Co\l[\frac{x_k-x_j}{1-x_j},\ell_k-\frac{1-x_k}{1-x_j}\ell_j\r]_{j<k\leq m}\r\rangle.
\een
\end{pro}
Additional elements on $\l\langle\Co\Big[x_j,\ell_j\Big]_{1 \leq j \leq m}\r\rangle$ are given below the proof and in Section \ref{sec:CGR}.
\begin{figure}[ht]
\centerline{\includegraphics{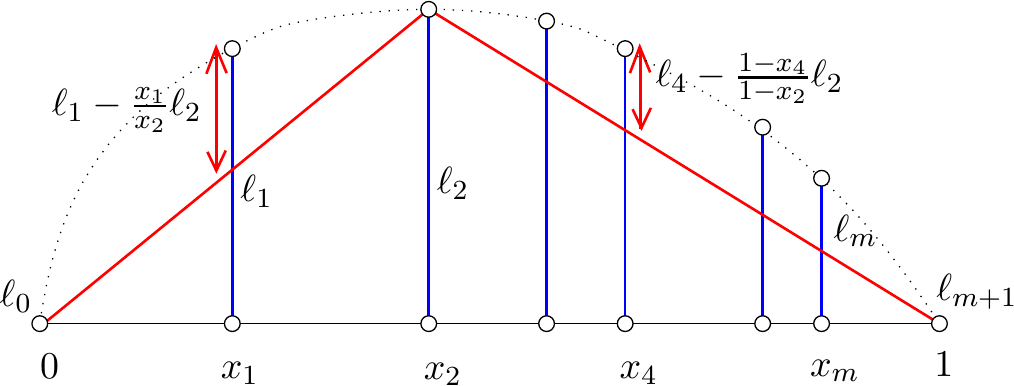}}
\captionn{\label{fig:comb2}Illustration of a term in the decomposition of $\l\langle\Co\big[x_,\ell_j\big]_{1 \leq j \leq m}\r\rangle$.} 
\end{figure}
\begin{proof} Assume $m\geq 2$ and take again the same notation as in Def. \ref{def:Pcomb}. For $t\in[0,1]$, consider
\[\Gamma_t:=\Co[x_j ,(1-t)\ell_j]\] 
so that $\Gamma_0=\Co[x_j,\ell_j]$ and $\Gamma_1=[0,1]$, and $\Gamma_t$ is obtained from $\Gamma_0$ by the affine map $A_t(x,y)=(x,y(1-t))$ which keeps $[0,1]$ unchanged and reduces the length of the teeth of the comb. As usual take some independent r.v $U_j$, where $U_j$ is uniform on $x_j+i[0,\ell_j]$ for $j\in\cro{1,m}$, and $U_0=0$, $U_{m+1}=1$. Of course 
\[(U_j, j \in \cro{1,m})\eqd(x_j+i\ell_j W_j, j \in \cro{1,m})\]
for $W_1,\cdots,W_{m}$ i.i.d. uniform on $[0,1]$.
Let $W^\star= \max\{W_j,j \in \cro{1,m}\}$ and $J$ the a.s. unique element $j$ such that $W_{J}=
W^\star$. By symmetry $J$ is uniform on $\cro{1,m}$ and conditionally on the event ${\sf Ev}_{j,w}:=\{(J,W_{J})=(j,w)\}$ the r.v. $W_k$, for $k\neq j$ are uniform on $[0,w]$, so that now, the r.v. $U_k,k\neq j$ are uniform in their corresponding tooth in $\Co[x_k,w\ell_k]_{1\leq k\leq m, k\neq j}$. Conditionally on ${\sf Ev}_{j,w}$,  $U[m+2]\in \CP_{m+2}$ iff  all the r.v. $U_k$ for $k \neq j$ are outside the triangle $U_0,U_{m+1},U_j$, and simultaneously $(U_0,\cdots,U_{j-1},U_{j}=x_j+iw \ell_j)\in \CP_{j+1}$ and $(U_j=x_j+iw \ell_j,U_{j+1},\cdots,U_{m+1})\in \CP_{m+2-j}$. Now,
\begin{itemize}
\item[$(a)$] the r.v. $U_k$ for $k \neq j$ are outside the triangle $U_0,U_{m+1},U_j$ if:
\begin{itemize}
\item[$\bullet$] for $k\leq j$ if $U_k$ belongs to $\bar{I}_k:=[\frac{x_k}{x_j} w\ell_j ,w\ell_k]$. This occurs (cond. on ${\sf Ev}_{j,w}$) with probability  $\frac{w\ell_k- \frac{x_k}{x_j} w\ell_j}{w\ell_k}=\frac{\ell_k- \frac{x_k}{x_j} \ell_j}{\ell_k}$ (which does not depend on $w$),
\item[$\bullet$] for  $k\geq j$, if $U_k$ belongs to $\bar{I}_k:=[w\frac{1-x_k}{1-x_j}\ell_j,w\ell_k]$. This occurs  (cond. on ${\sf Ev}_{j,w}$) with probability $\frac{w\ell_k-w\frac{1-x_k}{1-x_j}\ell_j}{w\ell_k}=\frac{\ell_k-\frac{1-x_k}{1-x_j}\ell_j}{\ell_k}$ (which does not depend on $w$).
\end{itemize}
\item[$(b)$] by $(a)$, for $k\neq j$, conditionally on ${\sf Ev}_{j,w}$, $U_k$ is uniform on $\bar{I}_k$. Then, we are in a situation where $U_0=0,U_1,\cdots,U_{j-1},U_{j}=x_j+iw \ell_j$ are uniform in an inclined comb (formed by the segments $\bar{I}_k,k \geq j$), as well as $U_{j}=x_j+iw \ell_j,U_{j+1},\cdots,U_{m+1}.$ The first inclined comb can be sent by an affinity on $\Co\l[\frac{x_k}{x_j},\ell_k-\frac{x_k}{x_j}\ell_j\r]_{1\leq k < j}$ and the second one on $\Co\l[\frac{x_k-x_j}{1-x_j},\ell_k-\frac{1-x_k}{1-x_j}\ell_j\r]_{j<k\leq m}$. 
\end{itemize}
We have now to put all the pieces together: under the conditioning ${\sf Ev}_{j,w}$, the probability that $U[m+2]\in\CP_{m+2}$ does not depend on $w$. When we integrate along the distribution of $W^\star$ this brings an ``extra factor'' of 1. Now, $J$ is uniform (and independent from $W^\star$) which brings a summation and a factor term $1/m$. The rest of the contributions are clear.
\end{proof}
Introduce
\ben\label{eq:K-toPC}
K[x_j,\ell_j]_{1 \leq j \leq m}:=\l(\prod_{k=1}^m \ell_k\r)\l\langle\Co\Big[x_j,\ell_j\Big]_{1 \leq j \leq m}\r\rangle.
\een
One sees that $K$ solves the following simpler recurrence: $K[ \varnothing]=1$, $K[x_1,\ell_1]=\ell_1$ and
\ben\label{eq:rec}
K[x_j,\ell_j]_{1 \leq j \leq m}& = & 
\frac1m \sum_{j=1}^m  \ell_j K\l[\frac{x_k}{x_j},\ell_k-\frac{x_k}{x_j}\ell_j\r]_{1\leq k < j}
K\l[\frac{x_k-x_j}{1-x_j},\ell_k-\frac{1-x_k}{1-x_j}\ell_j\r]_{j<k\leq m}.
\een
The first non trivial formula is $K[x_j,\ell_j]_{1\leq j \leq 2}$~:
\ben\label{eq:K2}
K[x_j,\ell_j]_{1\leq j \leq 2}&=&
\l(   \frac {\ell_1}{x_1(1-x_1)}   ( {\frac {\ell_2-\ell_1}{x_2-x_1}}-{\frac {\ell_2}{x_2-1}} ) +  \frac {\ell_2}{x_2(1-x_2)}   ( {\frac {\ell_1}{x_1}}-{\frac {\ell_1-\ell_2}{x_1-x_2}} )  \r)\\
\nonumber&\times& {x_1 ( x_2-x_1 )  ( 1-x_2 )}/2. 
\een
It is immediate that $K[x_j,\ell_j]_{1 \leq j \leq m}$ is a polynomial in the $\ell_j's$, with 
${\sf degree}\l(\prod_{j=1}^m\ell_j^{n_j}\r):=\sum {n_j}=m$ and coefficients in the fields $\mathbb{Q}(x_1,\cdots,x_n)$. And then \eref{eq:form} can be rewritten
\ben \label{eq:form2}
\bla {V}^{\beta[N+2]}[N+2] \bra&=&\int_{-L_0}^{L_0}\int_{-L_{N+1}}^{L_{N+1}}  \frac{g(\ell^+,\ell^-,u_0,u_{N+1})}{\prod_{j=0}^{N+1}w_j}du_{0}du_{N+1}
\een
where
\ben
 g(\ell^+,\ell^-,u_0,u_{N+1})=\sum_{A\subset\cro{1,N}}   K[x_j,\ell_j^+-u_j]_{j\in A}\,\, K[x_j,\ell_j^-+u_j]_{j\in \complement A}. \een
Hence $\bla  {V}^{\beta[N+2]}[N+2]\bra $ can be explicitly computed, as integrated a polynomial is just an exercise.  
The symmetry of the integration domain corresponding to the symmetry of $V^{\star}_0$ and $V^{\star}_{N+1}$ leads us to set
\ben\label{eq:GG}
G(\ell^+,\ell^-,u_0,u_{N+1}):=\sum_{(\varepsilon,\varepsilon')\in\{-1,1\}^2} g(\ell^+,\ell^-,\varepsilon u_0,\varepsilon' u_{N+1})
\een
and inflating the short notations $\ell^+,\ell^-,{V}^{\star}$ (introduced in \eref{eq:t1}, \eref{eq:t2},\eref{eq:t3}) we have 
\be
\bla {V}^{\beta[N+2]}[N+2] \bra= \frac{1}{\prod_{j=0}^{N+1}w_j}\int_{0}^{L_0}\int_{0}^{L_{N+1}}  {G(L+\lambda+\beta,L+\lambda-\beta,u_0,u_{N+1})}du_0du_{N+1}.
\ee
So Prop. \ref{pro:red}, as well as Theorem \ref{theo:ineq} appear to be a consequence of the following Proposition:
\begin{pro}\label{pro:last}
Under the hypothesis of Prop. \ref{pro:red}, for any $u_0,u_{N+1}\in[0,L_0]\times[0,L_{N+1}]$, the following map defined on  ${\sf Compa}({\lambda[N+2]})$
\[\beta[N+2]\mapsto G(L+\lambda+\beta,L+\lambda-\beta,u_0,u_{N+1})\]
reaches its maximum when $\beta[N+2]=0[N+2]$ and its minimum when $\beta[N+2]=\pm \lambda[N+2]$.
\end{pro}

\section{Optimization, and end of the proof of Theorem \ref{theo:ineq}:}
\label{sec:opti}

We have now enough information to proceed to the optimization of $\bla   {V}^{\beta[N+2]}[N+2] \bra$ for $N\in\{2,3\}$, or rather, to proceed to the optimization of $G(L+\lambda+\beta,L+\lambda-\beta,u_0,u_{N+1})$. 
In Section \ref{seq:AR} we provide a formula for $K$ as a  sum involving $2^m$ terms (choices of $(A,\complement A)$) each of them being a product of $m$ terms. This huge number of terms and their complexity of course makes of the optimization problem a difficult task. We tried to treat larger $N$ by multiple methods including convexity, recurrence, variation calculus, decomposition of $g$ by packing its terms.

In the case $N\in\{2,3\}$ the optimization can be achieved by brute force. When $N=2$ this provides an alternative proof to that of Blaschke. The case $N=3$ needs some important computations resources, difficult to handle without the assistance of a computer algebra system.
\begin{note}In the sequel, we will often have to prove that a linear function of $x_1$ (resp. $x_2$ or $x_3$) is positive on $A_1=[0,x_2]$ (resp. $A_2=[x_1,x_3]$, or $A_3=[x_2,1]$). We will say that a polynomial $P\in\mathbb{R}[x_1,x_2,x_3]$   is in ${\sf Lin}(x_j)$ if the degree of $P$ in $x_j$ is 1. To prove the positivity of such a polynomial on $A_j=[\underbar{A}_j,\bar{A}_j]$ we will only prove the positivity of the pair  $P(\underbar{A}_j),P(\bar{A}_j)$. Notice that this pair determines $P$ (which is 
$P=x_j\frac{(P(\underbar{A}_j)-P(\bar{A}_j))}{\underbar{A}_j-\bar{A}_j}+\frac{(\bar{A}_jP(\underbar{A}_j)-\underbar{A}_jP(\bar{A}_j))}{\bar{A}_j-\underbar{A}_j}$). In the sequel, often, we will describe $P$ by giving  $P(\underbar{A}_j),P(\bar{A}_j)$ only.
\end{note}

\subsection{Optimization: Case $n=4$ (that is $N=2$)}

We prove here Proposition \ref{pro:last} in the case $n=4$, that is $N=2$. We then fix, $L(x)=l_0+(l_1-l_0)x$ for some $(l_0,l_1)\in[0,+\infty)^2$, and $\beta[4]\in {\sf Comp}({\lambda[4]})$. with the formula we have given in the previous sections, we compute
\ben\label{eq:dif1}
G({\lambda}+L,\lambda+L,a,b) -G({\lambda}+\beta+L,\lambda-\beta+L,a,b)
=4\frac{\beta_2^2 x_1}{x_2}+4\frac{\beta_1^2 (1-x_2)}{1-x_1}\een
from what we see that $G({\lambda}+L,\lambda+L,a,b) \geq G({\lambda}+\beta+L,\lambda-\beta+L,a,b)$ in every case. Observe that the RHS in \eref{eq:dif1} does not depend on $(a,b)$, nor on $(l_0,l_1)$ which may seem strange at the first glance but can be understood even without computing $G$ (see Section \ref{sec:CGR}). Compute now 
\ben\label{eq:dif2}
G({\lambda}+L+\beta,\lambda-\beta+L,a,b) -G(2\lambda+L,L,a,b)
=4(\lambda_2^2-\beta_2^2)\frac{x_1}{x_2}+4(\lambda_1^2-\beta_1^2)\frac{1-x_2}{1-x_1}.
\een
When \eref{eq:pq} holds this is always non-negative. This suffices to see that Proposition \ref{pro:last} holds true when $n=4$, and then Theorem \ref{theo:main} in this case as explain along the paper.

\subsection{Optimization: Case $n=5$}

We prove here Proposition \ref{pro:last} in the case $n=5$, that is $N=3$, which again implies Theorem \ref{theo:main} in this case.

\subsubsection{Minoration}

We compute with a computer algebra system
\ben
D=G({\lambda}+L+\beta,\lambda-\beta+L,a,b) -G(2\lambda+L,L,a,b)\een
and prove its non negativity when $\beta[N+2]\in{\sf Comp}({\lambda[N+2]})$. 
The complexity of the task comes from the fact that the simplest formula we can find for $D$ is huge! Indeed,  $G(\ell^+,\ell^-,a,b)$ is a sum of 8 terms, each of them being a product of 3 linear form in the $\ell^+_j,\ell^-_j,a,b$. Moreover $\ell^+_j$ and $\ell^-_j$ are themselves sum of three terms. After expansion some cancellations occur but the number of remaining terms is still important. There are not obvious way to pack the terms, or to make a proof by recurrence that would permit to optimize $D$ for a given $N$ using the preceding ones. 

Moreover, if we did not make any mistake, $D$ is not non negative for all $\beta$ satisfying only $|\beta_j|\leq \lambda_j$ (for every $j$). It seems that the condition $|q_j|\leq p_j$ (for every $j$) is needed. We will then work in terms of $p_j$'s, $q_j'$s and will prove the positivity of $G$ on 
\[{\sf Compa}^{\star}({\lambda[N+2]})=\{\beta[N+2]:|q_j(\beta)|\leq p_j(\lambda), \beta_0=\beta_{N+1}=0\}.\] 
We then proceed to the change of variables described in \eref{eq:lambdavvp}, and prove the positivity of $D$ by writing $D$ as a polynomial with variables $p_j, (p_j+q_j), (p_j-q_j)$ and non negative coefficients in $\mathbb{Q}(x_1,x_2,x_3)$ when $0\leq x_1\leq x_2\leq x_3\leq 1$ . 
A polynomial that can written in such a form will be said to be in ${\sf Cone}(p,p-q,p+q)$.
An example is $(x_2-x_1)(p_1+q_1)+(x_3-x_1x_2)(p_2-q_2)$.

Taking again $L(x)=l_0+(l_1-l_0)x$ for some $(l_0,l_1)\in[0,+\infty)^2$, and expanding $D$ the result appears to have degree 1 in $l_0$ and in $l_1$ and degree 0 in $a$ and $b$ (explanations are given in Section \ref{sec:CGR}). Hence $D$ can be uniquely written under the form
\[D=4l_0D_0+4l_1D_1+4D_2\]
where $D_1,D_0,D_2$ are polynomial in $x,p,q$. We search to prove that each $D_i\in{\sf Cone}(p,p-q,p+q)$.
We will then prove this by proving that $D_2$ can be written $\sum_{1\leq i,j,k\leq 3} c_{k,i,j} p_k(p_i+q_i)(p_j-q_j)$ with some non negative $c_{k,i,j}$. There are several solutions, one of them satisfies the needed conditions. The solution is as follows: for any $(k,i,j)$, $c_{k,i,j}=c_{k,j,i}$, and
\be
c_{1,1,1}&=&2\,{ {x_1^3 ( 1-x_3 )  ( 1-x_2 )  ( x_3-x_1 ) }/{x_3}}\\
c_{1,1,2}&=&P_0(x_1) (1-x_3)x_1^2/x_3\\
c_{1,1,3}&=&2\,{ {x_1^2x_2 ( 1-x_3 ) ^2 ( x_3-x_1 ) }/{x_3}}\\
c_{1,2,2}&=&2P_1(x_3) \,{\frac { ( 1-x_3 ) x_1}{x_3 ( 1-x_1 ) }}\\
c_{1,2,3}&=&2P_2(x_3)\,{\frac { ( 1-x_3 ) ^2x_1}{x_3 ( 1-x_1 ) }}\\
c_{1,3,3}&=& 2P_3(x_3)\,{\frac { ( 1-x_3 ) ^2x_1}{ ( 1-x_2 )  ( 1-x_1 ) }}
\ee
with $P_0(x_1)\in{\sf Lin}(x_1)$, $P_0(0)=x_2^2+x_2x_3-2x_2^2x_3>0$, $P_0(x_2)=2x_2(1-x_2)(x_3-x_2)>0$.\\
$P_1(x_3),P_2(x_3),P_3(x_3)\in{\sf Lin}(x_3)$ and $P_1(x_2)=x_2 ( 1-x_2 )  ( -x_1x_2-x_1+2\,x_2 )  ( x_2-x_1 ) >0$, $P_1(1)=( 1-x_2 )  (  ( x_1x_2-x_1 ) ^2+ ( 1-x_1 )  ( x_2-x_1 ) x_2 ) >0$,
 $P_2(1)=(x_1x_2-x_2)^2 +(x_1-x_2)^2 >0$, $P_2(x_2)=x_2 ( -x_1x_2-x_1+2\,x_2 )  ( x_2-x_1 ) >0 $, 
 $P_3(1)= x_2 ( 1-x_1 ) ^2 ( 1-x_2 ) >0$, and $P_3(x_2)= ( 1-x_2 )  ( x_1x_2+x_1-2\,x_2
 )  ( x_1-x_2 ) >0$.

\be
c_{2,1,1}&=&{ {x_1^2 ( 1-x_3 ) }{}}P_4(x_1)/x_3\\
c_{2,1,2}&=&2\,{\frac { ( 1-x_2 ) x_1^2 ( 1-x_3 ) }{x_3 ( 1-x_1 ) }} P_5(x_1)\\
c_{2,1,3}&=&2\,{\frac {x_1^2 ( 1-x_3 ) ^2}{x_3 ( 1-x_1 ) }} P_5(x_1)\\
c_{2,2,2}&=& {\frac { ( 1-x_3 ) x_1}{x_3 ( 1-x_1 )  ( x_3-x_1 ) }} P_0(x_1)P_5(x_1)\\
c_{2,2,3}&=&2\,{\frac { ( 1-x_3 ) ^2x_1x_2}{x_3 ( 1-x_1 ) }}P_5(x_1)\\
c_{2,3,3}&=&{\frac { ( 1-x_3 ) ^2x_1}{1-x_1}} P_4(x_1)
\ee
with $P_4(x_1)\in{\sf Lin}(x_1)$, and $P_4(0)=x_2 ( -2\,x_2x_3-x_2+3\,x_3 )>0$, $P_4(x_2)=2\,x_2 ( 1-x_2 )  ( x_3-x_2 ) >0$.
with $P_5(x_1)\in{\sf Lin}(x_1)$, and $P_5(0)=x_2 ( -x_2x_3-x_2+2\,x_3 ) >0$, $P_5(x_2)=2\,x_2 ( 1-x_2 )  ( x_3-x_2 ) >0$.

We need to prove the positivity of the last series $c_{3,i,j}$.

\be
c_{3,1,1}&=&2\,{\frac {x_1^2 ( 1-x_3 ) }{x_2x_3}}P_6(x_1)\\
c_{3,1,2}&=&2\,{\frac {x_1^2 ( 1-x_3 ) }{x_3 ( 1-x_1 ) }} P_7(x_1)\\
c_{3,1,3}&=&2\,{\frac { ( 1-x_3 ) ^2 ( 1-x_2 ) ( x_3-x_1 ) x_1^2}{1-x_1}}\\
c_{3,2,2}&=&2\,{\frac { ( 1-x_3 ) x_1}{x_3 ( 1-x_1 ) }}P_8(x_1)\\
c_{3,2,3}&=&{\frac { ( 1-x_3 ) ^2x_1}{1-x_1}}P_0(x_1)\\
c_{3,3,3}&=&2\,{\frac { ( 1-x_3 ) ^3 ( x_3-x_1 ) x_1x_2}{1-x_1}}
\ee
with $P_6(x_1),P_7(x_1),P_8(x_1)\in{\sf Lin}(x_1)$, and $P_6(0)=x_2{x_3}^2 ( 1-x_2 )  >0$, $P_6(x_2)=x_2 ( -x_2x_3-x_2+2\,x_3 )  ( x_3-x_2 )$, $P_7(0)= ( x_2x_3-x_3 ) ^2+ ( x_2-x_3 ) ^2  >0$, $P_7(x_2)= ( 1-x_2 )  ( -x_2x_3-x_2+2\,x_3 )  ( x_3-x_2 ) >0$, and $P_8(0)=x_2 (  ( x_2x_3-x_2 ) ^2+x_3 ( 1-x_2 )  ( x_3-x_2 )  )$, $P_8(x_2)=x_2 ( 1-x_2 )  ( -x_2x_3-x_2+2\,x_3 )  ( x_3-x_2 ) >0$. Hence, we have established that $D_3>0$.
We now prove with the same method that $D_1\in{\sf Cone}(p+q,p-q)$. One finds that 
$D_1$ can be written $\sum_{1\leq i,j\leq 3} c_{i,j} (p_i+q_i)(p_j-q_j)$ for $c_{1,3}=c_{3,1},c_{2,1}=c_{1,2},c_{3,2}=c_{2,3}$ and 
\be
c_{1,1}&=& 2\,{\frac {x_1^2}{x_2x_3}} P_6(x_1) \\
c_{1,2}&=& 2\,{\frac {x_1^2}{ ( 1-x_1 ) x_3}}P_7(x_1) \\
c_{1,3}&=& 2\,{\frac {x_1^2 ( 1-x_3 )  ( 1-x_2 )  ( x_3-x_1 ) }{1-x_1}}\\
c_{2,2}&=& 2\,{\frac {x_1}{ ( 1-x_1 ) x_3}}P_8(x_1) \\
c_{2,3}&=& {\frac { ( 1-x_3 ) x_1}{1-x_1}}P_0(x_1) \\
c_{3,3}&=& 2\,{\frac {x_1x_2 ( 1-x_3 ) ^2 ( x_3-x_1 ) }{1-x_1}}
\ee
which proves that $D_1>0$. By symmetry $D_0>0$ too.

\subsection{Majoration}
\label{sec:maj}
We compute with a computer algebra system
\[D:=G({\lambda}+L,\lambda+L,a,b) -G({\lambda}+\beta+L,\lambda-\beta+L,a,b)\]
we find that 
\[D=l_0 f_1+l_1 f_2 +f_3\]
where the terms $f_1$, $f_2$, $f_3$ are polynomial, linear in $\lambda$, quadratic in $\beta$ (and do no depends on $l_0$, $l_1$ nor on $(a,b)$). 
Here the formula are small enough to be written down. 
One finds:
\be
f_1/8&=&(\,\beta_1\beta_2(x_3-1)+\,\beta_1x_2\beta_3+\,
\beta_2\beta_3x_1)-\,{\frac {
\beta_3 ( \beta_3x_1x_2+\beta_1x_2+\beta_
2x_1-x_2\beta_3 ) }{x_3}}\\
&+&\,{\frac {{\beta_1}^2 ( 1-x_3 )  ( 1-x_2 ) }{1-x_1}}+\,{\frac {{\beta_2}^2 ( 1-x_3 )  ( 1-x_1 ) }{1-x_2}},\ee
and $f_2$ can be obtained from $f_1$ by a change of variables (by symmetry); $f_3$ appears to be
\be
f_3/4&=&\beta_1 (\beta_1\lambda_3+2\,\beta_3\lambda_1  -\beta_1\lambda_1) +2\,{\frac {x_1{\beta_2}^2 ( \lambda_2x_3-\lambda_2+\lambda_3 ) }{x_2}}+{\frac { ( \beta_1-\beta_3 ) ^2 ( \lambda_1-\lambda_3 )  ( x_1-x_2 ) }{x_1-x_3}}\\
&-&{\frac {\beta_3 ( 2\,\beta_1\lambda_3x_2+2\,\beta_2\lambda_3x_1-\beta_3\lambda_1x_2-\beta_3\lambda_2x_1
+\beta_3\lambda_3(x_1-x_2 )) }{x_3}}+2\,{\frac {{\beta_2}^2 ( 1-x_3 )  ( \lambda_1-\lambda_2x_1) }{1-x_2}}\\
&+&{\frac {\beta_1 ( \beta_1(\lambda_1x_2-\lambda_1x_3+\lambda_2x_3+\lambda_3x_2-\lambda_2-\lambda_3)+2\,\beta_2\lambda_1(1-x_3)  +2\,\beta_3\lambda_1(1-x_2) ) }{x_1-1}}.
\ee
Again as explained in Section \ref{sec:CGR}, $a$ and $b$ play not role at all. 
Now, $f_1$ and $f_3$ are quadratic forms in the $\beta_i$'s. Again, $D$ seems not positive on the set of $\beta[5]$ satisfying only $|\beta_j|\leq \lambda_j$. We work again on ${\sf Compa}^\star$ after making the change of variables \eref{eq:lambdavvp}.
 
\subsubsection{Positivity of $f_1$}
We then find that the quadratic form $f_1$ written in terms of $q$ can be written
\[f_1=(q_1,q_2,q_3) M (q_1,q_2,q_3)^{t}\]
where $M= 4\frac{(1-x_3)(x_3-x_1)x_1}{x_3} N$ and $N$ is the symmetric matrix with coefficients:
\[N=\left[ \begin {array}{ccc} 2\,x_1 ( 1-x_2 ) &{\frac {P_0(x_1)}{x_3-x_1}}&2\, ( 1-x_3 ) x_2\\ \noalign{\medskip}{\frac {P_0(x_1)}{x_3-x_1}}&2\,{\frac {P_1(x_3)}{ ( x_1-x_3 ) x_1 ( x_1-1 ) }}&2\,{\frac { ( 1-x_3 ) P_2(x_3)}{ ( x_3-x_1 ) x_1 ( 1-x_1 ) }}
\\ \noalign{\medskip}2\, ( 1-x_3 ) x_2&2\,{\frac { ( 1-x_3 ) P_2(x_3)}{ ( x_3-x_1 ) x_1 ( 1-x_1 ) }}&2\,{\frac {P_3(x_3)\, ( 1-x_3 ) x_3}{ ( x_3-x_1 ) x_1 ( x_
2-1 )  ( x_1-1 ) }}\end {array} \right]
\]
 
For any matrix $M$ denote by $M[k]$ the extracted matrix $(M_{i,j})_{1\leq i,j\leq k}$.
To prove the non negativity of $f_1$ it suffices to prove the positivity of $\det(M[k])$.

We find $\det(N[1])=2(1-x_2)x_1>0$, 
\[\det(N[2])\frac{(x_3-x_1)^2}{(x_1-x_2)^2}(1-x_1)=g(x_1)\in {\sf Lin}(x_1)\]
with $g(0)=( -2\,x_2x_3+x_2+x_3 )  ( -2\,x_2x_3-x_2+3\,x_3 ) >0$ and $g(x_2)= ( 1-x_2 )  ( -4\,x_2x_3+x_2+3\,x_3 )  ( x_3-x_2 ) >0$. 
Last
\[\det(N[3])\frac{(x_3-x_1)^3(1-x_2)(1-x_1)x_1}{x_3(1-x_3)(x_2-x_1)^2(x_3-x_2)^2}  
= g(x_3)\in {\sf Lin}(x_3)
\]
with  $g(x_2)=2\, ( 1-x_2 )  ( 3\,x_1x_2-x_1-2\,x_2 )  ( x_1-x_2 )>0$, $g(1)=6\,x_2 ( 1-x_1) ^2 ( 1-x_2 ) >0$, we deduce the fact that $\det(N[3])>0$. Hence $M$ is definite positive, and then $f_1$ is positive except when $(q_1,q_2,q_3)=0$ in which case it is 0.

\subsubsection{Positivity of $f_3$}

Since $f_3$ is linear in $p$ and quadratic in $q$, we write
\[f_3=\sum_{i=1}^3 p_i (q_1,q_2,q_3) M^{(i)}  (q_1,q_2,q_3)^{(t)}.\]
The $p_i$'s being non negative, is suffices to prove the positivity of $\det(M^{(i)}[j])$ for $i\in\{1,2,3\}$, $j\in\{1,2,3\}$. Since $p_1$ and $p_3$ plays the same role (by symmetry with respect to the line $x=1/2$, only the cases $i\in\{1,2\}$ have to be treated.

We have for $N:=M^{(1)}\frac{x_3}{4(1-x_3)^2}$, 
\[N=\left[ \begin {array}{ccc} 2\,{\frac {{x_1}^3 ( x_3-x_1 )  ( 1-x_2 ) }{1-x_3}}&{\frac 
{P_0(x_1)\,{x_1}^2}{1-x_3}}&2\,{x_1}^2x_2 ( x_3-x_1 ) \\ \noalign{\medskip}{\frac {P_0(x_1)\,{x_1}^2}{1-x_3}}&2\,{\frac {P_1(x_3)\,x_1}{ ( 1-x_3 )  ( 1-x_1 ) }}&2\,{\frac {P_2(x_3)\,x_1}{1-x_1}}\\ \noalign{\medskip}2\,{x_1}^2x_2 ( x_3-x_1 ) &2\,{\frac {P_2(x_3)\,x_1}{1-x_1}}&2\,{\frac {P_3(x_3)\,x_1x_3}{ ( 1-x_2 )  ( 1-x_1 ) }}\end {array} \right]\]
Hence,
\[\det(N[1])=-2\,{\frac {{x_{{1}}}^{3} \left( x_{{2}}-1 \right)  \left( x_{{1}}-x_{
{3}} \right) }{x_{{3}}-1}}>0\]
\[\det(N[2])={\frac { \left( x_{{1}}-x_{{2}} \right) ^{2}}{ \left( x_{{3}}-x_{{1}}
 \right) ^{2} \left( 1-x_{{1}} \right) }} g(x_1)\in {\sf Lin}(x_1)\]
  with
$g(x_2)= ( 1-x_2 )  ( 4\,x_2x_3-x_2-3\,x_{3
} )  ( x_2-x_3 )>0$, $g(0)= ( 2\,x_2x_3-x_2-x_3 )  ( 2\,x_2x_
3+x_2-3\,x_3 ) >0$.
Finally,
\[\det(N[3])=2\,{\frac { \left( x_{{1}}-x_{{2}} \right) ^{2} \left( x_{{2}}-x_{{3}}
 \right) ^{2}x_{{3}} \left( x_{{3}}-1 \right) }{ \left( x_{{1}}-x_{{3}
} \right) ^{3}x_{{1}} \left( x_{{1}}-1 \right)  \left( x_{{2}}-1
 \right) }}
 g(x_3)\in {\sf Lin}(x_3)\]
with
\[g(x_2)= ( 1-x_2 )  ( -3\,x_1x_2+x_1+2\,x_2 )  ( x_2-x_1 )>0 ,
g(1)=3\,x_2 ( 1-x_1 ) ^2 ( 1-x_2 ) >0\]

Remains the matrix 

\[M^{(2)}= 4\left[ \begin {array}{ccc} {\frac {P_4(x_1)\,{x_1}^2 ( 1-x_3 ) }{x_3}}&2\,{\frac {{x_1}^2 ( 1-x_3 )  ( 1-x_2 ) P_5(x_1)}{ ( 1-x_1 ) x_3}}&2\,{\frac {{x_1}^2 ( 1-x_3 ) 
^2P_5(x_1)}{ ( 1-x_1 ) x_3}}\\ \noalign{\medskip}2\,{\frac {{x_1}^2 ( 1-x_3 )  ( 1-x_2 ) P_5(x_1)}{ ( 1-x_1 ) x_3}}&{\frac { ( 1-x_3 ) x_1P_0(x_1)\,P_5(x_1)}{x_3 ( 1-x_1 )  ( x_3-x_1 ) }}&2\,{\frac {x_1
 ( 1-x_3 ) ^2x_2P_5(x_1)}{ ( 1-x_1 ) x_3}}\\ \noalign{\medskip}2\,{\frac {{x_1}^2 ( 1-x_3 ) ^2P_5(x_1)}{ ( 1-x_1 ) x_3}}&2\,{\frac {x_1 ( 1-x_3 ) ^2x_2P_5(x_1)}{ ( 1-x_1 ) x_3}}&{\frac {P_4(x_1)\, ( 1-x_3 ) ^2x_1}{1-x_1}}\end {array} \right] \]

Along almost the same lines, $\det(M^{(2)}[1])=4{\frac {x_1^2 ( 1-x_3 ) }{x_3}}P_4(x_1)>0$,
 
\[\det(M^{(2)}[2])= 16 {\frac { ( 1-x_3 ) ^2x_1^3 ( x_2-x_1 ) ^2}{{x_3}^2 ( x_3-x_1 ) ( 1- x_1 ) ^2}} g_1(x_1)P_5(x_1)\]
with $g_1(x_1)\in {\sf Lin}(x_1)$ and 
\[g_1(x_2)= ( 1-x_2 )  ( -4\,x_2x_3+x_2+3\,x_3 )  ( x_3-x_2 ) >0\]
\[g_1(0)=( -2\,x_2x_3+x_2+x_3 )  ( -2\,x_2x_3-x_2+3\,x_3 ) >0\]
\[\det(M^{(2)}[3])=3\times 64 \,{\frac { ( x_2-x_1 ) ^2 ( x_2-x_{3
} ) ^2{x_1}^{4} ( x_3-1 ) ^{4}}{{x_3}^{
2} ( 1-x_1 ) ^2 ( x_3-x_1 ) }} P_4(x_1)P_5(x_1)>0.\]

\section{Algebraic considerations}
\label{sec:CGR}

\subsection{Elements on the function $K$}
\label{seq:bin-struct}
Recall Prop. \ref{pro:comb-dec} and \eref{eq:K-toPC}. Formula \eref{eq:rec} implies that $K$ possesses a binary tree structure: choose uniformly a pivot $j$, splits the points in two parts (those with indices $<j$, those with indices $>j$) and iterate in both parts. In Fig. \ref{fig:comb2}, see the first decomposition done according to $j=2$. The triangle with vertices $v_0,v_j,v_{m+1}$ with $v_k=x_k+i \ell_k$ plays an important role.  When one expands the computation in each of the subtrees, one sees that at the second step\\
-- in the part at the left of $v_j$, $v_0$ is still at the left, but now $v_j$ plays now the role of the right boundary (as did before $v_{m+1}$),\\
-- in the part at the right of $v_j$, $v_{m+1}$ is still at the right, but now $v_j$ plays now the role of the left boundary (as did before $v_{0}$). \par 
Let $z[m+2]$ be a sequence of different points in the plane where $z_j=(x_j,\gamma_j)$, 
for some $x[m+2]\in\ND{m+2}, x_0=0, x_{m+1}=1$. Let us call {\it directed triangle } 3-tuples of the form $t=(z_{j_1},z_{j_2},z_{j_3})$ with $0\leq j_1<j_2<j_3\leq m+1$. We call {\it triangulation } of $z[m+2]$ a set $T$ of triangles satisfying the following conditions: \\
$\bullet$ $\#T=m$ (this is the number of different triangles in $T$),\\
$\bullet$ the triangles are non crossing : if $t=(z_{j_1},z_{j_2},z_{j_3})$ and $t'=(z_{j_1'},z_{j_2'},z_{j_3'})$ are two different triangles in $T$ then either:\\
-- $\{j_1,j_2,j_3\}$ is included in one of the intervals $\cro{0,j'_1}$, $\cro{j'_1,j'_2}$, $\cro{j'_2,j'_3}$, $\cro{j'_3,m+1}$\\
-- or $\{j_1',j_2',j_3'\}$ is included in one of the intervals $\cro{0,j_1}$, $\cro{j_1,j_2}$, $\cro{j_2,j_3}$, $\cro{j_3,m+1}$.\par
\begin{figure}[ht]
\centerline{\includegraphics[width=11cm]{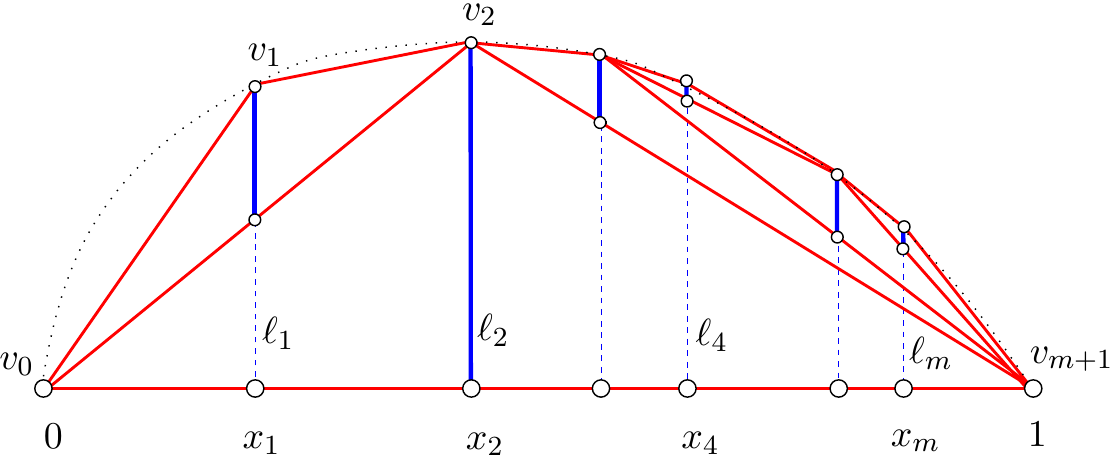}}
\captionn{\label{fig:comb3} Illustration of a triangulation $t$.} 
\end{figure}
\noindent Denote by $\Tri(z[m+2])$ the set of triangulations of $z[m+2]$. When $T\in\Tri(z[m+2])$, the set of central points $\{ z_{j_2} \,|\, (z_{j_1},z_{j_2},z_{j_3}) \in T\}$ equals $\cro{1,m}$. Let $t=(z_{n_1},z_{n_2},z_{n_3})$ be a triangle. Let
\ben\label{eq:qt}
q_t(x[m+1],\gamma[m+1])&=&\l[{\gamma_{n_2}- \l(\gamma_{n_1}+\frac{x_{n_2}-x_{n_1}}{x_{n_3}-x_{n_1}}(\gamma_{n_3}-\gamma_{n_1})\r)}\r]/(n_3-n_1-1)
\een
be the length of the vertical segment from $z_{n_2}$ to the segment $[z_{n_1},z_{n_3}]$ divided by $n_3-n_1-1$. Set 
\ben
{K}^\star(x[m+2],\gamma[m+2])=\sum_{T\in \Tri(z[m+2])} \prod_{t\in T}  q_t(x[m+2],\gamma[m+2]).
\een
We have the following combinatorico-geometrical representation of $K$~:
\begin{theo}\label{theo:repq}
 For any $x[m+2]\in \ND{m+2}$, $x_0=0,x_{m+1}=1$, any $\ell[m+2]\in[0,+\infty)^{m+2}$ such that $\ell_0=\ell_{m+1}=0$, we have
\ben\label{eq:dqsd}
K[x_j,\ell_j]_{1\leq j \leq m}={K}^\star(x[m+2],\ell[m+2]).
\een
\end{theo}
\begin{proof} Expand \eref{eq:rec}. \end{proof}

\begin{lem} If $L_i=l_0+x_i(l_{1}-l_0)$ is affine (in this sense), for any sequence $\gamma[m+2]$,
\ben\label{eq:dqsd2}
K^\star(x[m+2],\gamma[m+2]+L[m+2]) =K^\star(x[m+2],\gamma[m+2]).
\een
\end{lem}
\begin{rem}\label{rem:kkstar} What is the difference between $K$ and $K^\star$~? By definition, $K[x_j,\ell_j]_{j \in A}$ depends on 4 extra quantities that are $x_0=0,x_{N+1}=1,\ell_0=0,\ell_{N+1}=0$. For $K^{\star}(x[m+2],\ell[m+2])$ these conditions are not assumed,  
and $K^{\star}$ is defined even when $\ell_0$ or $\ell_{N+1}$ are not 0. Hence, when $\ell$ is linear in the sense $\ell_j=a+x_j(b-a)$, we do have $K^\star(x[m+2],\ell[m+2])=0$ but $K[x_j,\ell_j]_{1\leq j \leq m}$ is not 0 (in general), since the computation of this quantity is done assuming $\ell_0=0,\ell_{m+1}=0$. What it still true is that $q_t(\ell)$ is 0 when the triangle $t$ does not contain $0$ or $m+1$. In the summation giving $K$, when $\ell$ is linear only remains the contributions of the triangulations in which each triangle is incident to 0 or $m+1$.
\end{rem}

\subsection{Alternative representation of $K[x_j,\ell_j]_{1\leq j \leq m}$}
\label{seq:AR}
Instead of taking the sum on all subtrees, we can make a summation on all permutations. For any $\sigma$ permutation in ${\cal S}\cro{1,m}$, set
\be
R_j(\sigma) & = & \min\{ \sigma_i~: i <j,\sigma_i>\sigma_j \} \vee 0,\\
L_j(\sigma) & = & \max\{ \sigma_i~: i <j,\sigma_i<\sigma_j \} \wedge (m+1). 
\ee 
Then one can prove 
\begin{lem}\label{eq:Ks} We have 
\ben
{K}^\star(x[m+2],\gamma[m+2])&=&\frac1{m!}\sum_{\sigma \in {\cal S}(\cro{1,m})} \prod_{j=1}^m 
\w{q}_{\sigma,i}(x[m+2],\gamma[m+2])
\een
where
\ben
\w{q}_{\sigma,i}(x[m+2],\gamma[m+2]):=
\l( \frac{\gamma_{\sigma_j}-\gamma_{L_j(\sigma)}}{x_{\sigma_j}-x_{L_j(\sigma)}}- \frac{\gamma_{R_j(\sigma)}-\gamma_{\sigma_j}}{x_{R_j(\sigma)}-x_{\sigma_j}} \r) \frac{(x_{\sigma_j}-x_{L_j(\sigma)})(x_{R_j(\sigma)}-x_{\sigma_j})}{x_{R_j(\sigma)}-x_{L_j(\sigma)}},
\een
and $x_0=0,x_{m+1}=1$. 
\end{lem}
\begin{proof} The proof is more or less the same as that appearing above, in Section \ref{seq:bin-struct}. In this section, we used the fact that $K$ owns a binary structure, when one observes separately the two parts lying apart of the pivot (dissection property of $K$, coming from Prop. \ref{pro:comb-dec}). But instead of considering ``independently'' what happens in this two parts, we can make the iteration on the collection of remaining segments (at time $k$, $k$ segments have been ``taken''). Conditional on the set of indices $A$ of these segments, the remaining points $U_j$ are uniform in these segments which have the form $x_j+i[y_j+\ell_j'w_j]$ for some $(y_j,\ell_j')$. We still make a decomposition taking the maximum of the $w_j$ (which is uniform in the set of remaining indices), and iterate. 
Putting all pieces together, we can see that Lemma \ref{eq:Ks} holds.
\end{proof}
\subsection{Simplifications in the computations of $G$}

The representation of $K$ given in Theorem \ref{theo:repq} allows one to see that $K[x_j,\ell_j]_{1\leq j \leq m}$ is a sum of products of quantities, the $q_t(\ell)$'s, linear in $\ell$. As already said in Remark \ref{rem:kkstar} for every triangle $t=(n_1,n_2,n_3)$ with $n_1>0,n_3<m+1$, $q_t(\ell)=0$ when $\ell$ is linear (except at the border). In the computation we are doing, we need to compute quantities of the form 
 $F= G(\lambda+L,\lambda+L,a,b)-G(\lambda+\beta+L,\lambda-\beta+L,a,b)$ which involves quantities of the form $K[x_j,L_j+u_j+\lambda_j+\beta_j]_{j\in A}$. Hence, the $q_t$ involved are of the form $q_t(L+u+\lambda+\beta)=q_t(L)+q_t(u)+q_t(\lambda)+q_t(\beta)$ and $L(x)=l_0+x(l_1-l_0)$. When one expands everything, since $u$ and $L$ are linear, the only contributions of $a$ and $b$ come from the triangles that are incident to the points 0 and $m+1$. The resulting big sum, is then a sum over the triangulations, of the products over each triangulations of the triangle associated values $q_t(\ell)$, with in each triangle a sum over $\ell=\lambda$, $\pm \beta$ (depending on the case), $u$ and $L$ (only when $t$ is adjacent to 0 or to $m+1$). Moreover, our decomposition formula shows that the degree of $F$ in $l_0,l_1,a,b,\lambda,\beta$ is $n$ (where the degree of a term $\prod v_i^{k_i}$ in these variables is $\sum k_i$). 

Besides $F$ is even in $\beta$. Hence when one expands everything, in terms of $q_t$, remains only a sum of product of $q_t$'s in which all terms cancel except those involving a positive even number of $q_T(\beta)$, some $q_T(L)$ incident to the boundary (by Remark \ref{rem:kkstar}), an even number of $q_T(u)$, incident to the boundary (by Remark \ref{rem:kkstar}); since $q_t(u)$ itself is linear in $a,b$.  Developing everything, since we sum on $(\varepsilon a,\varepsilon'b)$ for $(\varepsilon,\varepsilon')\in\{-1,1\}^2$, remains only the terms even in $a$ and in $b$. \\
Notice also that the terms involving no $\beta_j$'s cancels. 

When $n=4$, that is $N=2$, these conditions put together imply that huge cancellations arise: remains only terms that depends on $\beta_j$'s. When $n=5$, remains only terms quadratic in $\beta$, linear in $\lambda$ and in $L$. 
When $n=6$, remains a polynomial much more complex: it contains terms with degree 2, 4 in the $\beta_j$, with coefficients with degree 1 or 2 in $\lambda$, 0 or 2 in $a$ and $b$, 0,1, or 2 in $l_0,l_1$.

\bibliographystyle{plain}

\end{document}